\documentclass[11pt]{amsart}
\usepackage{latexsym}
\usepackage{graphicx, color, epsfig, verbatim, subfigure, epsf} 
\usepackage{amssymb, amsmath,mathrsfs}
\usepackage[left=2.5cm,right=2.5cm,top=2.5cm, bottom=2.5cm ]{geometry}
\usepackage{graphicx}
\usepackage{tikz}
\usepackage{pgfplots}

\usepackage{float}

\usepackage{setspace}
\usepackage[pagebackref,hypertexnames=false, colorlinks, citecolor=red, linkcolor=red]{hyperref}

\newcommand{\McC}{\raise.5ex\hbox{c}}

\newcommand{\C}{\mathbb{C}}

\newcommand{\Gg}{\mathcal{G}}

\newtheorem{theorem}{Theorem}[section]

\newtheorem*{theorem*}{Theorem}
\newtheorem*{conjecture*}{Conjecture}

\newtheorem*{corollary*}{Corollary}
\newtheorem{proposition}[theorem]{Proposition}
\newtheorem*{proposition*}{proposition}

\def\bb{\begin{color}{blue}}
\def\bg{\begin{color}{green}}
\def\br{\begin{color}{red}}
\def\eg{\end{color}}
\def\er{\end{color}}
\def\eb{\end{color}}

\theoremstyle{remark}
\newtheorem{remark}[theorem]{Remark}
\newtheorem{problem}[theorem]{Problem}

\newtheorem{definition}[theorem]{Definition}
\newtheorem{example}[theorem]{Example}

\begin{document}

\title[Inner functions in multiply connected domains]{On the concept of inner function in Hardy and Bergman spaces in multiply connected
domains}

\author[B\'en\'eteau]{Catherine B\'en\'eteau}
\address{Department of Mathematics, University of South Florida, Tampa, FL 33620, U.S.A.}
\email{benetea@usf.edu}
\author[Fleeman]{Matthew Fleeman}
\address{Department of Mathematics, Baylor University, Waco, TX 76798 , U.S.A.}
\email{MatthewFleeman@baylor.edu}
\author[Khavinson]{Dmitry Khavinson}
\address{Department of Mathematics, University of South Florida, Tampa, FL 33620, U.S.A.}
\email{dkhavins@usf.edu}
\author[Sola]{Alan A. Sola}
\address{Department of Mathematics, Stockholm University, 106 91 Stockholm, Sweden}
\email{sola@math.su.se}

\thanks{To Stephen J. Gardiner on the occasion of his 60th birthday.}

\begin{abstract}
We discuss the notion of an inner function for spaces of analytic functions in multiply connected domains in $\mathbb{C}$, giving a historical overview and comparing several possible definitions. We explore connections between inner functions, zero-divisors for Hardy spaces and Bergman spaces, and weighted reproducing kernels. After recording some obstructions and negative results, we suggest avenues for further research and point out several open problems.
\end{abstract}

\def\bb{\begin{color}{blue}}
\def\bg{\begin{color}{green}}
\def\br{\begin{color}{red}}
\def\eg{\end{color}}
\def\er{\end{color}}
\def\eb{\end{color}}



\maketitle

\section{Introduction}\label{Intro}
In this paper, we study certain special classes of bounded functions in multiply connected domains: inner functions, furnishing isometric or contractive divisors for analytic function spaces. The problems we are interested in have been explored extensively in the context of function spaces in the unit disk, going back to the early 20th century. Recall that a sequence 
\[\mathcal{Z}=\{z_j\}_{j=1}^{\infty}\subset \Gg\]
is said to be a zero set for a Banach space $\mathscr{X}$ of analytic functions in a domain $\Gg \subset \mathbb{C}$ if there exists a function $f\in \mathscr{X}$ such that $f(z_j)=0$ for all $j=1,2, \ldots$ and $f(z)\neq 0$ for $z\in \Gg\setminus \mathcal{Z}$. We adhere to the usual convention of taking multiplicities into account by allowing repetitions in  $\mathcal{Z}$. It is often desirable to be able to ``factor out" potential zeros of a given function $f \in \mathscr{X}$, and, if possible, to describe the generator(s) of zero set-based subspaces of $\mathscr{X}$ that are invariant with respect to multiplication by functions analytic in a neighborhood of $\Gg$, or with respect to bounded analytic functions.

To give a concrete example and to set the scene for what will follow, let us briefly revisit the well-understood case of the Hardy spaces $H^p$ in the unit disk $\mathbb{D}$. For $p>0$ fixed, the space $H^p$ consists of analytic functions $f\colon \mathbb{D}\to \mathbb{C}$ satisfying the growth restriction
\[\|f\|^p_p=\lim_{r\to  1^-}\frac{1}{2\pi}\int_{-\pi}^{\pi}|f(re^{i\theta})|^pd\theta<\infty.\]
When $p\geq 1$ this furnishes a norm that turns  $H^p$ into a Banach space.
Consider a sequence $\{z_j\}$ of points in $\mathbb{D}\setminus \{0\}$ that satisfies the Blaschke condition $\sum_{j=1}^{\infty}(1-|z_j|)<\infty$. Then $\{z_j\}$ is a zero set for the Hardy space and it is well-known (see \cite{Durbook,Fbook,Hoffbook}) that one can associate a bounded analytic function $B_\mathcal{Z}$, called a Blaschke product, to such a sequence, and this function acts as an isometric divisor for any function $f\in H^p$ that vanishes on $\mathcal{Z}$. That is,  $\|f/B_{\mathcal{Z}}\|_{H^p}
=\|f\|_{H^p}$ if $f(z_j)=0$, $j=1, 2, \ldots$. In particular, Blaschke products are examples of \emph{inner} functions in the Hilbert space $H^2$, a class of functions that admits several useful characterizations. The most direct condition for $f\in H^2$ to be inner is that $|f(e^{i\theta})|=1$ at almost every point on the unit circle $\mathbb{T}=\{|z|=1\}$. Equivalently, a function $f \in H^2$ is inner if $\|f\|=1$ and the orthogonality condition 
$\langle f, z^jf\rangle_{H^2}=\int_{\mathbb{T}}f(\zeta) \overline{\zeta^jf(\zeta)}|d\zeta|=0$ is satisfied for $j=1,2,\ldots$. Yet another approach via extremal problems is quite helpful. Namely, if $\mathcal{Z}$ is a $H^2$-zero set, then the aforementioned Blaschke product solves the extremal problem $\sup\{\mathrm{Re}f(0)\colon \|f\|_{H^2}\leq 1, \,f(z_j)=0, \, j=1,2,\ldots\}$

More recently, a largely parallel theory of zero-divisors was developed for the Bergman space in the disk in the 90s, see \cite{Hbook,DSbook}. In this case, zero divisors $G$ are no longer isometric; rather, dividing a function by an associated zero divisor decreases the Bergman norm. Nevertheless, contractive divisors are inner functions for the Bergman space $A^2$ and such functions can again be characterized in terms of orthogonality conditions and as solutions to extremal problems \cite{DSbook,Hbook}.

The situation in the multiply connected setting is considerably more complicated for several reasons. For instance, the harmonic conjugate of a harmonic function in a multiply connected domain is not in general single-valued, and constructions relying on explicit computations with orthonormal bases or reproducing kernels become significantly more challenging. One idea that comes to mind for reducing the multiply connected setting to the simply connected one is to introduce nice cuts that avoid zero sets, but this cannot work. Indeed, if we construct a contractive or isometric divisor for the Bergman or Hardy cases, respectively, in the resulting simply connected domain this will, of course, provide the needed norm estimates, but though analytically continuable across the cuts, we will in general be left with a multi-valued function in the original domain.

The purpose of the present work is to survey the current state of affairs for domains of multiple connectivity, as we understand them, and to present some new contributions as well as possible directions for further investigations.

We begin by setting down some notation. Let $\Omega$ be a bounded finitely connected domain in $\mathbb{C}$, with boundary given by
\[\partial \Gg=\Gamma=\bigcup_{j=1}^n\gamma_j,\]
where each $\gamma_j$ is a rectifiable Jordan curve. We reserve $\gamma_1$ for the outer boundary curve. We denote harmonic measure for $\Omega$ with fixed base point $z_0\in \Gg$ by 
$\omega=\omega(z_0,\ \cdot, \ \Omega)$, and we write $\omega_j$ for the harmonic measure of each of the boundary components $\gamma_j$, that is, for $j=1,\ldots, n$, 
\[\Delta \omega_j=0 \quad \textrm{in}\quad \Gg, \quad \textrm{and}\quad \omega_j\ _{\big|_{\gamma_k}}=\delta_{jk}.\]
We sometimes use the shorthand notation $\omega_{z_0}$ to denote the measure $\omega(z_0, \cdot, \Gg)$.
If the $\gamma_j$ are smooth we have the representation
\[d\omega(z_0, \  \cdot,\ \Omega)=-\frac{1}{2\pi}\frac{\partial g}{\partial n}(\cdot, z_0)ds,
\]
where $g(\cdot, z_0)$ denotes the Green's function of $\Omega$ with pole at $z_0$, $\partial/\partial n$ is the derivative in the direction of the outward normal of $\Gamma$, and $ds$ denotes arclength measure.

\begin{figure}[h!]
\centering
\includegraphics[width=0.44 \textwidth]{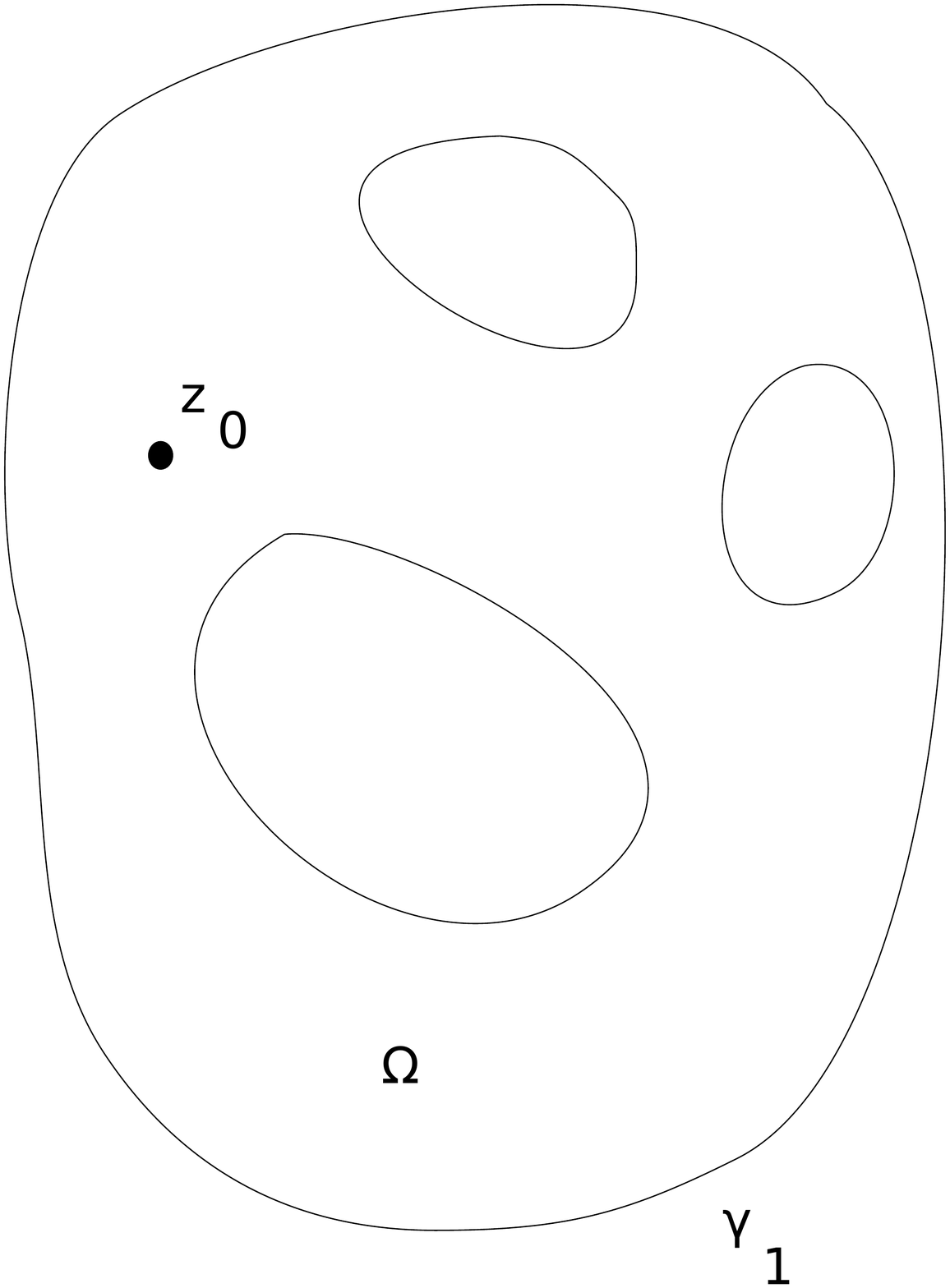}
  \caption{\textsl{A smoothly bounded multiply connected domain $\Omega$ with basepoint $z_0$ and outer boundary curve $\gamma_1$}.}
  \label{faveplots}
\end{figure}

We shall make use of approximations of a given domain $\Gg$: a regular exhaustion is a sequence $\{\Gg_j\}_{j=1}^{\infty}$ of $n$-connected domains with analytic boundaries $\{\Gamma_j\}$, such that $z_0 \in \Gg_j$, $\overline{\Gg_j}\subset \Gg_{j+1}$, and $\Gg=\bigcup_{j=1}^{\infty}\Gg_j$. See \cite{Fbook,GMbook} for background material on harmonic measure and related topics.

We are interested in certain scales of Banach spaces of analytic functions in multiply connected domains: Hardy, Bergman, and Smirnov spaces. For a fixed $p>0$, the Hardy space $H^p(\Gg)$ consists of analytic functions $f\colon \Omega\to \mathbb{C}$ with the property that, for a regular exhaustion $\{\Gg_j\}$ of $\Gg$,
\begin{equation}
\|f\|^p_{H^p(\Omega)}=\lim_{j\to \infty}\int_{\Gamma_j}|f(\zeta)|^pd\omega(z_0,\zeta, \Omega_j)<\infty.
\label{hardynorm}
\end{equation}
When $p\geq 1$, this expression defines a norm.
As usual, $f^*$ will be taken to denote the boundary values of the function $f$ initially defined in $\Gg$. Membership of a function in $H^p(\Gg)$ is preserved under change of basepoint $z_0$.
See \cite[Chapters 3 and 4]{Fbook} and \cite{Durbook} for background information on Hardy norms, including proofs of existence of boundary 
values $\omega$-almost everywhere. We remark that $H^p(\Gg)$ can also be defined in terms of harmonic majorants \cite{P,R}: $f\in H^p(\Gg)$ if there is a harmonic function $u\colon \Gg\to \mathbb{R}$ satisfying 
\[|f(z)|^p\leq u(z), \quad z\in \Gg.\] 
This definition lends itself to extensions to arbitrary domains in $\mathbb{C}$, and to Riemann surfaces \cite{Heibook}.

We shall also need some close relatives of Hardy spaces, namely the Smirnov spaces. We say that an analytic function $f\colon \Gg \to \C$ belongs to the Smirnov class $E^p(\Gg)$, $p>0$, if 
\begin{equation}
\|f\|^p_{E^p(\Gg)}=\lim_{j\to \infty}\int_{\Gamma_j}|f(\zeta)|^pds
<\infty.
\label{smirnovnorm}
\end{equation}
Here, we are integrating with respect to arclength measure $ds$ on $\Gamma_j$.
More generally, we write $E^p(\rho  ds)$ for Smirnov spaces associated with a non-negative weight function $\rho \colon \overline{\Gg} \to [0,\infty)$.
It can be shown that the modulus of any function $f\in H^p(\Gg)$ ($p \geq 1$) can be recovered from its non-tangential boundary values using Poisson integrals \cite{Fbook,Durbook},
\[f(z)=\int_{\Gamma}f^*(\zeta)d\omega(z, \zeta, \Gg).\]
A Smirnov function $f\in E^p(\Gg)$, $p\geq 1$, can be represented as a Cauchy integral of its non-tangential boundary values $f^*$,
\[f(z)=\frac{1}{2\pi}\int_{\Gamma}\frac{f^*(\zeta)}{\zeta-z}d\zeta, \quad z\in \Gg.\]
We remark that $H^p(\Gg)$ and $E^p(\Gg)$ coincide as sets if $\Gamma$ is smooth, but are in general different in rough domains. See, 
for instance, \cite{Durbook,SKhavTum1,SKhavTum2,SKhavTum3}, for more in-depth discussions of Hardy and Smirnov spaces.

We will find it useful to also consider the Nevanlinna class $N(\Gg)$ and the Smirnov class $N_{+}(\Gg)$. We recall that an analytic function $f\colon \Omega \to \mathbb{C}$ is said to belong to the Nevanlinna class if
\[\limsup_{j\to \infty}\int_{\Gamma_j}\log^+|f(\zeta)|\, d\omega(z_0, \zeta, \Omega_j) < \infty, \]
or, equivalently, if $\log^+|f|$ has a harmonic majorant in $\Gg$. If $f$ has the additional property that the family
\[\left\{\int_{\Gamma_j}\log^+|f(\zeta)|d\omega_j(z_0, \zeta, \Omega_j)\right\}_j\]
is uniformly integrable with respect to harmonic measure,
we say that $f$ belongs to the Smirnov class $N_{+}(\Gg)$. See \cite{Khav1,SKhavTum3} for more on this. Note, cf. \cite{SKhavTum3}, that $N\supset N_{+}\supset H^p$, and that $N\supset E^p$ for all $p>0$. In certain domains with rough boundaries, so-called non-Smirnov domains, $E^p$ is in general not contained in $N_+$. For $p<q$, we have $H^p\supset H^q$ and $E^p\supset E^q$.

Nevanlinna functions (and hence Hardy functions) have good non-tangential limit properties, but this is lost when we move to Bergman spaces \cite{Bellbook,Bergbook}. For $p\geq 1$, the Bergman space $A^p(\Gg)$ is the Banach space of analytic functions $f$ in $\Gg$ satisfying the norm boundedness condition
\begin{equation}
\|f\|_{A^p(\Gg)}^p=\int_{\Gg}|f(z)|^pdA(z)<\infty.
\label{bergmannorm}
\end{equation}
Here, $dA=\pi^{-1}dxdy$ denotes normalized area measure in the plane.
In this note, we shall mostly deal with the Hilbert spaces $H^2(\Gg)$ and $A^2(\Gg)$.

We mention in passing that it is possible, and sometimes useful, to define function spaces in a multiply connected domain $\Gg$ in terms of uniformizing maps by viewing the unit disk $\mathbb{D}$ as a covering surface. See \cite{Fbook} for more on this point of view; we shall not use this approach in the present paper.

We denote the collection of rational functions whose poles are off $\overline{\Gg}=\Gg \cup \Gamma$ by $R(\Gg)$. It is known \cite{Durbook,Fbook,Bellbook} that $R(\Gg)$ is dense in all the function spaces we consider, provided the boundary $\Gamma$ is smooth enough. We shall also 
consider the set $A(\Gg)$ consisting of analytic functions in $\Gg$ that are continuous on $\overline{\Gg}$; we note that $R(\Gg)$ is uniformly dense in $A(\Gg)$, again, see \cite[Chapter 4]{Fbook}.

The purpose of this note is to explore the notion of inner function and zero divisor for spaces of analytic functions in a multiply connected domain $\Gg$. Suppose we are given a zero set for a function space $\mathscr{X}$ in $\Gg$, say, the Hardy or the Bergman space. Does there exist a function $G$
in $\mathscr{X}$ that divides out zeros, that is to say, if $f  \in \mathscr{X}$ vanishes on $\mathcal{Z}$, do we have the estimate $\|f/G\|_{\mathscr{X}}\leq \|f\|_{\mathscr{X}}$ (or even equality)? If this does not hold, can we hope for a norm comparison involving a universal constant that only depends on $\Gg$ and $\mathscr{X}$?

Assuming such functions exist, do they belong to a natural class of inner functions, defined in terms of some suitable extremal problem or in terms of orthogonality conditions?

As we shall see, these questions are subtle and we are not able to give complete answers to all of them. Several difficulties arise when we leave the class of simply connected domains, both conceptual and technical.  A fundamental problem, as was mentioned earlier, is that of harmonic conjugation: given a harmonic function $h$ in $\Gg$ whose properties we understand, such as the Green's function, we often wish to form an analytic function $f=h+i\widetilde{h}$, where $\widetilde{h}$ denotes the harmonic conjugate of $h$. If $\Gg$ is simply connected, this is straight-forward, but in the multiply connected setting, the harmonic conjugate is no longer single-valued in general: $\widetilde{h}$ typically has periods around the holes of $\Gg$. For instance, $h=\log |z|$ is harmonic in any annulus $\{z\in \mathbb{C}\colon r<|z|<1\}$, but $\widetilde{h}=\arg z$ has period $2\pi$ around the inner circle $r\mathbb{T}$. Hence, we are faced with the challenge of removing such periods through appropriate modifications without losing whatever desirable properties $f=h+i\widetilde{h}$ may possess. Another notorious obstruction is that certain canonical objects like reproducing kernels are no longer zero-free in the multiply connected setting, unlike in the disk. Typically, the higher the connectivity, the more zeros appear in the kernel function, and this means that care has to be taken when dividing by kernels.

At the technical level, it is often useful to be able to perform explicit computations with orthonormal bases and to have closed expressions for canonical objects like Poisson kernels and reproducing kernels. In the disk, this is often possible, and formulas for reproducing kernels and single-point divisors are attractive. By contrast, even in very simple multiply connected domains like the annulus or circular domains, explicit expressions are often lacking, or require a more cumbersome analysis of special functions. 

To keep technicalities to a minimum, we shall usually assume that the boundary curves $\Gamma=\bigcup_j\gamma_j$ are analytic. The kinds of questions we consider in this paper could also be posed for function spaces on Riemann surfaces. There, yet another technical issue arises from topological considerations involving the handles of the surfaces. We shall mostly avoid this level of generality, and refer the reader to \cite{W,VZ,Heibook,SKhav1,SKhav2} for guidance to the Riemann surface setting.

\section{Definitions and basic properties of inner functions in multiply connected domains}\label{InnerDefs}
We shall now review several possible definitions of  ``inner functions'' in a multiply connected domain, and compare and contrast these definitions. We begin with some historical remarks. In the context of the unit disk, classical $H^p$-inner functions have played a prominent role in most treatments of function spaces in the unit disk, going back to the factorization theorems of R. Nevanlinna and V.I. Smirnov (see \cite{Durbook,Fbook,Hbook,Golbook}). For subclasses of the Nevanlinna class in the multiply connected setting, the survey by S.Ya. Khavinson and G.Ts. Tumarkin \cite{SKhavTum4} is a good starting point. The extension of the notion of Hardy spaces to multiply connected domains and Riemann surfaces can be traced back to Parreau \cite{P} and Rudin \cite{R}. These spaces have subsequently been studied by numerous authors, see for instance the references in \cite{Fbook,Khav1,Khav2}. A first construction of ``Blaschke products" in a finitely connected domain was suggested by Zmorovich \cite{Z}, and the convergence of these products was established by Tamrazov \cite{Tam}. (These works, as well as several imporant papers by Dunduchenko and Kas'yanyuk \cite{DK1,DK2}, were published in Ukrainian or Russian and unfortunately did not become widely known in the West.) A convenient construction of an analog of the Schwarz kernel, and consequently of ``Blaschke products" was presented by Coifman and Weiss \cite{CW}. Further generalizations, which will be useful in what follows, were obtained by D. Khavinson \cite{Khav1,Khav2} (see also the references therein), Kuzina \cite{K}, and S.Ya. Khavinson \cite{SKhav1,SKhav2}, the latter in the context of finite Riemann surfaces. 

For Bergman spaces in general domains in $\mathbb{C}$ we refer the reader to \cite{Bergbook,Bellbook}. The term ``inner function" in Bergman spaces in the unit disk was coined by Korenblum, following breakthroughs by Hedenmalm and by Duren, Khavinson, Shapiro, and Sundberg. These more recent developments are detailed in the books \cite{Hbook,DSbook}; more recent results on Bergman spaces in multiply connected domains can be found in \cite{AH}. As far as the authors are aware, a concept of inner functions for Bergman spaces in multiply connected domains has not been developed in a systematic way so far.

As was mentioned earlier, in the classical setting of the Hardy space of the unit disk, a function $f\in H^2(\mathbb{D})$
is said to be inner (or $H^{2}$- inner to be precise) if $\left|f^*(\zeta)\right|=1$ for almost every $\zeta \in \mathbb{T}$. 
Equivalently, $f\in H^2(\mathbb{D})$ can be declared to be inner precisely if
\begin{equation}
\|f\|_{H^2(\mathbb{D})}=1 \quad \textrm{and}\quad \langle z^jf, f\rangle_{H^2(\mathbb{D})}=0 \quad \textrm{for}\quad j=1, 2, \ldots \,.  
\label{canadinnerdef}
\end{equation}
See \cite{BKLSS,BFKSS} and the references therein for systematic treatments of this approach to inner functions for general Hilbert function spaces, as well as \cite{CMR} and references cited there for the more general Banach space case.

In the multiply connected setting, the following definition has evolved over the years (cf. \cite[Chapter 4]{Fbook} and the papers listed above).
\begin{definition}\label{multinnerconstants}
Let  $\Omega\subset \mathbb{C}$ be a bounded domain whose boundary consists of finitely many analytic curves $\gamma_1, \ldots, \gamma_n$. We say that $f\in H^2(\Omega)$ is inner if
\[\|f\|_{H^2(\Gg)}=1\quad \textrm{and}\quad 
|f^*(\zeta)|_{\big|_{\gamma_j}}=c_j \quad \textrm{for}\quad  j=1,\ldots, n\]
where $c_1,\ldots, c_n$ are positive constants.
\end{definition}
We review a construction (see \cite{VZ,Khav1}) that lends some credence to the idea that the above is a natural generalization of inner function.
\begin{example}[Generalized Blaschke products]\label{blaschkeex}
Let $z_0\in \Omega$ be fixed, and let $g(z, z_0)$ be Green's function for $G$ with pole at $z_0$. Suppose $\mathcal{Z}=\{z_j\}_{j=1}^{\infty}$ is a sequence of points in $\Omega$ such that the generalized Blaschke condition \cite{Fbook,CW,Khav1}
\[\sum_{j=1}^{\infty}g(z_j,z_0)<\infty\]
is satisfied: this means that $\mathcal{Z}$ is a zero set for $H^2(\Gg)$. 

Let $p(z)=g(z,z_0)+i\widetilde{g}(z,z_0)$ be the multi-valued analytic completion of Green's function. It is natural to try to exponentiate $p$ to get an inner function, but in order to obtain a single-valued analytic function, we need to cancel the periods of $\widetilde{g}$.
As is explained in \cite{Khav1}, there exists a non-unique choice of constants $\lambda_1,\ldots, \lambda_n\in \mathbb{R}$ such that the generalized Blaschke product
\begin{equation}
B_{\mathcal{Z}}(z)=\exp\left(-\sum_{j=1}^{\infty}p(z_j,z_0)+\sum_{k=1}^n\lambda_k(\omega_k+i\widetilde{\omega}_k)\right)
\label{blaschkeformula}
\end{equation}
is a single-valued analytic function. Furthermore, see \cite{Khav1,Khav2}, we have $B_{\mathcal{Z}}(z_j)=0$ for $j=1,2,\ldots$, and 
\[|B_{\mathcal{Z}}|=c_j=e^{\lambda_j}\quad \omega-\textrm{almost everywhere on}\quad \gamma_j.\]
\end{example}
Similarly, see \cite{Khav2}, one can construct singular inner functions.
\begin{example}[Singular inner functions]\label{singularex}
Let $\mu$ be a non-positive measure with $\mu \perp \omega_{z_0}$. We can construct a singular inner function by setting
\begin{equation}
S_{\mu}(z)=\exp\left[\frac{1}{2\pi}\int_{\Gamma}\left(\frac{\partial g}{\partial n_{\zeta}}(z,\zeta)+i\frac{\partial \widetilde{g}}{\partial n_{\zeta}}(z,\zeta)\right)d\mu(\zeta)+\sum_{j=1}^{n}\lambda_j(\omega_j(z)+i\widetilde{\omega}_j(z))\right],
\label{singularformula}
\end{equation}
by selecting constants $\lambda_1,\ldots ,\lambda_n$ in such a way that $S_{\mu}$ becomes single-valued. 

We again have $|S_{\mu }(\zeta)|=c_j$ almost everywhere with respect to $\omega_{z_0}$, but now $S_{\mu}$ is a non-vanishing inner function in $\Gg$.
\end{example}
It can be shown \cite{CW,Khav1} that any inner function in $H^2(\Gg)$ can be written as a product of a Blaschke product and a singular inner function, and this decomposition is unique up to invertible factors. We caution that there are non-constant invertible inner functions in the multiply connected setting: the functions $z^k$ ($k\neq 0$) in the annulus $\{r<|z|<1\}$ furnish examples of this.

Using the generalized inner functions from Definition \ref{multinnerconstants}, one can formulate an extension
of Beurling's theorem on invariant subspaces of $H^2(\Gg)$, once the notion of ``$z$-invariant'' is appropriately modified. This provides some justification for calling functions satisfying Definition \ref{multinnerconstants} $H^2(\Gg)$-inner.  We shall discuss this in Section \ref{InvariantSubs}.

We will now investigate to what extent the definition of inner function in Definition \ref{multinnerconstants} squares with an inner product condition as in \eqref{canadinnerdef}. In a Hilbert function space $\mathscr{X}$ in the multiply connected setting, a naive approach would be to replace the 
orthogonality condition \eqref{canadinnerdef} with the condition that 
\begin{equation}
\|f\|_{\mathscr{X}}=1 \quad \textrm{and}\quad \langle f, rf\rangle_{\mathscr{X}}=0, \quad r\in R(\overline{\Gg})\setminus \textrm{span}_{\C}\{1\}.
\label{multiconninnerdef}
\end{equation}
In this paper, we shall focus on the cases $\mathscr{X}=H^2(\Gg) $ and $\mathscr{X}=A^2(\Gg)$, the Hardy and Bergman spaces respectively.

We begin with the simplest possible case.
\begin{example}
Let $\Omega=\left\{ z:\,r<\left|z\right|<1\right\} $ be an annulus
centered at the origin. Since $\Omega$ has analytic boundary, an equivalent
norm for $H^2(\Omega)$ is obtained by replacing harmonic measure by arclength measure $ds$, and considering $L^2(\Gamma, ds)$-integrals. 

In order for $f$ to be inner in the sense of Definition \ref{multinnerconstants}, we must have that $\left|f\right|^{2}\in\mathrm{span}\left\{ \omega_{1},\omega_{2}\right\} $,
where $\omega_{1}, \omega_2$ are the harmonic measures on the outer and inner circles, that is,
\[
\omega_{1}(z, \mathbb{T}, \Gg)=\frac{\log\frac{\left|z \right|}{r}}{\log\frac{1}{r}}\quad \textrm{and}\quad \omega_{2}(z, r\mathbb{T},\Gg)=\frac{\log\left|z\right|}{\log r}.
\]
We wish to determine $\left\{ \omega_{1},\omega_{2}\right\} ^{\perp}$
in $L^{2}(\Gamma,\,ds)$. But the functions $\frac{\partial\omega_{j}}{\partial n}$,
where $n$ is the outward pointing normal, are also locally constant on
$\gamma_{j}$ since $\omega_{j}$ are radial. Hence $\left\{ \omega_{1},\omega_{2}\right\} ^{\perp}$
coincides with $\left\{ \frac{\partial\omega_{1}}{\partial n},\frac{\partial\omega_{2}}{\partial n}\right\} ^{\perp}$,
which is equal to $\mathrm{span}\left\{ \mathrm{Re}\,z^{n}:\:n\in\mathbb{Z}\right\} $.
Since $\left|f\right|^{2}$ is already real-valued, we are led to
the following proposition.
\end{example}
\begin{proposition}
Let $\Omega=\left\{ z\in\mathbb{C}\colon \,r<\left|z\right|<1\right\} $. The function $f\in H^2(\Gamma, ds)$ is inner 
in the sense that $f$ is constant a.e. with respect to arc-length
on $\mathbb{T}$ and on $r\mathbb{T}=\left\{ z\in \mathbb{C}\colon \left|z\right|=r\right\} $ 
if and only if 
\[
\int_{\Gamma}\left|f\right|^{2}z^{n}ds=0\qquad n\in\mathbb{Z}\setminus\{0\},
\]
or, in other words, if 
\[
\left\langle z^{n}f,f\right\rangle =0\qquad n\in\mathbb{Z}\setminus\{0\}.
\]
\end{proposition}

Inspired by this simple example, let us return to the general setting. Unfortunately, the results are somewhat disappointing.

%
\begin{proposition}\label{Schottkyprop}
Let $\Omega\subset \mathbb{C}$ be a bounded domain with boundary $\Gamma=\cup_{j=1}^{n}\gamma_j$, where each $\gamma_j$ is a closed analytic curve.

Then if $f \in H^2(\Gamma, ds)$ satisfies the orthogonality condition \eqref{multiconninnerdef}, we have
\[|f(z)|^2=1+\sum_{j=1}^{n-1}\lambda_j \frac{\partial \omega_j}{\partial n}, \]
for some real constants $\lambda_1, \ldots,  \lambda_{n-1}$.

In particular, a function satisfying condition \eqref{multiconninnerdef} is not necessarily locally constant on $\Gamma$ and hence not necessarily inner in the sense of Definition \ref{canadinnerdef}.
\end{proposition}
\begin{proof}
Suppose $f\in H^2(\Gg)$ satisfies \eqref{multiconninnerdef}. Then
\begin{equation}
\int_{\Gamma}(|f|^2-1)\overline{r}(z)ds=0
\label{Rortho}
\end{equation}
for any function $r\in R(\Gg)$. Since $|f|^2$ is real, this implies that $(|f|^2-1)\perp \mathrm{Re}[R(\Gg)]$, the set of real parts of elements of $R(\Gg)$.

However, the description of measures $\Gamma$-orthogonal to $\mathrm{Re}[R(\Gg)]$ is well-known, see \cite[Theorem 2.3]{Fbook}. From the proof, we extract that the $n$ linearly independent measures spanning $\mathrm{Re}[R(\Gg)]^{\perp}$ are of the form
\[\frac{\partial \omega_j}{\partial n}ds, \quad j=1, \ldots, n-1.\]
Thus, in light of \eqref{Rortho} we conclude that $|f(\zeta)|^2=1+\sum_{j}\lambda_j\frac{\partial \omega_j}{\partial n}(\zeta)$, as desired.
\end{proof}
Examining the proof of Proposition \ref{Schottkyprop}, it becomes apparent that the simple orthogonality condition in the case of $H^2(\Gamma, ds)$ of the annulus is a result of the fact that the functions $\partial \omega_j/\partial n$ are themselves constant on the boundary components. Bearing this in mind, and working out the argument for the standard Hardy space $H^2(\Gg)$ of the annulus (with harmonic measure instead of arclength), we find that $f\in H^2(\Gg)$ satisfies the orthogonality condition \eqref{multiconninnerdef} in $L^2(\Gamma, d\omega_{z_0})$ precisely if 
\[|f(\zeta)|^2= 1+\lambda_j\frac{\frac{\partial \omega_j}{\partial n}(\zeta)}{\frac{\partial g(\zeta, z_0)}{\partial n_\zeta}}, \quad \zeta \in \Gamma, \quad j=1, 2.\]
for some real $\lambda_1, \lambda_2$.  This identity does not imply that $|f(\zeta)|_{\Big| \Gamma}$ is locally constant on boundary components. Thus, even in this case being inner is not the same as satisfying the orthogonality condition if we insist upon the standard norm involving $\omega_{z_0}$.

Switching to the standard Hardy space $H^2(\Gg)$ and reworking the proof above, we conclude that $f\in H^2(\Gg)$ satisfying the orthogonality condition implies that $|f|^2$ is a linear combination of Schottky functions \cite[Chapter 6]{Golbook},
\[|f(\zeta)|^2=1+\sum_{1}^{n-1}\lambda_j s_j(\zeta), \quad \textrm{where}\quad s_j=\frac{\partial \omega_j}{\partial n_{\zeta}}(\zeta)\big/\frac{\partial g}{\partial n_{\zeta}}(\zeta,z_0).\]

 We take a brief detour here to point out a connection with an interesting and non-trivial question in free boundaries, 
known as the Vekua problem \cite{BKhav}. This is the overdetermined problem of characterizing the domains $\Omega$ for which there exists a non-zero solution to the overdetermined boundary value problem ($j=1,2)$
\begin{equation}
\left\{\begin{array}{cc}\Delta u=0 & \textrm{in}\quad  \Omega\\
u=c_j, \,\, \partial u/\partial n=d_j & \textrm{on} \quad \gamma_j \end{array}\right.,
\label{vekua}
\end{equation}
where $\{c_j\}_{j=1}^2$ and $\{d_j\}_{j=1}^2$ are constants. It has been shown \cite[Theorem 5.5]{KhavSV} that among doubly connected domains satisfying mild smoothness conditions, only annuli admit solutions. It was  conjectured in \cite{KhavSV} that in domains of connectivity three and higher, the Vekua problem has no solutions at all. 

Returning to our previous discussion, we note that the computation carried out for arclength measure in the annulus is valid only in that setting, and fails in all other doubly connected domains. Turning to Bergman spaces $A^p(\Gg)$ in multiply connected domains, we do not know of any working definition of inner function that is as simple as that in Definition \ref{multinnerconstants} or the orthogonality condition \eqref{canadinnerdef}, while at the same time leading to a wide enough class of functions to act as zero-divisors and generators. Indeed, in the now classical setting of $A^p(\mathbb{D})$, the analogous requirement that $|G(\zeta)|=1$ for almost every $\zeta \in \mathbb{T}$ does not characterize a satisfactory class of ``inner functions". Instead, the usual approach to Bergman-inner functions in the disk is in terms of certain extremal problems. We shall discuss this approach next.

\section{Inner functions, extremal problems, and weighted reproducing kernels}\label{ExtremalProbs}
Another way of identifying inner functions, or at least Blaschke products
(in Hardy spaces) and contractive divisors (in Bergman spaces), is
to view them as solutions to certain extremal problems. For a fixed basepoint $z_0\in \Gg$ and a point $z_1\in \Gg\setminus\{z_0\}$ 
we pose the following two extremal problems for $\mathscr{X}=H^2(\Gg)$ or $A^2(\Gg)$: 
\begin{equation}
\sup\left\{ \left|f(z_{0})\right|\colon f(z_{1})=0,\,\left\Vert f\right\Vert _{\mathscr{X}}=1\right\} \label{eq:extprob1}
\end{equation}
and 
\begin{equation}
\inf\left\{ \left\Vert g\right\Vert _{\mathscr{X}}\colon g(z_{0})=1,\,g(z_{1})=0\right\} .\label{eq:extprob2}
\end{equation}
These problems are actually equivalent since if $F$ is an extremal function
for (\ref{eq:extprob1}), then $G(z)=\frac{F(z)}{F(z_{0})}$
is an extremal problem for (\ref{eq:extprob2}).

We now carry out a standard variational argument. Consider functions $f\in A(\Gg)$
such that $f(z_{0})=0.$ Then if $G$ is an extremal function for
(\ref{eq:extprob2}), and $\lambda\in\mathbb{C},$ then the function $G_{\lambda}:=G\cdot(1+\lambda f)$
is a contender. In particular, $G_{\lambda}(z_{1})=0$, and $G_{\lambda}(z_{0})=1$.
Then since $G$ is extremal for (\ref{eq:extprob2}), we have that
\begin{align*}
\|G\|^{2}_{\mathscr{X}}&\leq \| G_{\lambda}\|^{2}_{\mathscr{X}}\\&=\left\langle G(1+\lambda f),G(1+\lambda f)\right\rangle_{\mathscr{X}} \\
&=\|G\|^{2}_{\mathscr{X}}+2\mathrm{Re}\left\langle G\lambda f,G\right\rangle_{\mathscr{X}} +\mathcal{O}\left(\lambda^{2}\right).
\end{align*}
Therefore we must have $\left\langle Gf,G\right\rangle_{\mathscr{X}} =0$, since otherwise we
could find a $\lambda$ such that $\| G_{\lambda}\|_{\mathscr{X}}^{2}<\|G\|_{\mathscr{X}}^{2}$.
Taking any $F\in A(\Omega),$ and letting $f:=F(z)-F(z_{0})$, we arrive at the condition 
\begin{equation}
\left\langle G\left(F-F(z_{0})\right),G\right\rangle_{\mathscr{X}}=0.
\label{eq:OrthoFact}
\end{equation}
Without loss of generality we may also assume that $\| G\|_{\mathscr{X}} =1$.

We now specialize to the case $\mathscr{X}=H^2(\Gg)$. Then (\ref{eq:OrthoFact}) implies that, for all $F\in A(\Omega)$,
\begin{equation}
\int_{\Gamma} F\left|G\right|^{2}d\omega_{z_{0}}=F(z_{0}).\label{eq:ReproFact}
\end{equation}
Hence, we deduce that
\[
\int_{\Gamma}\left(\left|G\right|^{2}-1\right)ud\omega_{z_{0}}=0
\]
for all $u=\mathrm{Re}F$, where $F\in A(\Omega)$. This requirement unfortunately does not imply that the extremal $G$ is locally constant on boundary components: as in the previous section, we only deduce that $|G|^2-1$ can be written as a linear combination of $n-1$ generically non-constant functions.

We now turn to the Bergman space, that is, the case $\mathscr{X}=A^2(\Gg)$. The condition \eqref{eq:OrthoFact} now implies that
\[\int_{\Gg}(|G(z)|^2-1)u(z) \,dA(z)=0,\]
again for all $u=\mathrm{Re}F$, where $F\in A(\Omega)$. Superficially, this condition looks very similar to the definition of a Bergman-inner function in the setting of the unit disk, see \cite[Chapter 3]{Hbook} and \cite{DSbook}. The difference between the simply connected and multiply connected cases
is that in the simply connected case, the space $\mathrm{Harm}(\overline{\Gg})$
of functions harmonic in $\Gg$ and continuous in $\overline{\Gg}$ simply consists of real parts of functions in $A(\Gg)$, whereas in the multiply connected case $\mathrm{Harm}(\overline{\Gg})=\mathrm{Re}[A(\Gg)]\oplus \mathscr{N}$,
where again $\dim \mathscr{N}=n-1$. Unlike in the Hardy case, a concrete description of the space $\mathscr{N}$ is not immediately apparent for general $\Gg$. In the the annulus $\Gg=\{r<|z|<1\}$, one checks that $\mathscr{N}=\mathrm{span}\{\log |z|-c_0\}$, where $c_0=2\pi \int_{r}^1r\log r dr$. In general, $\mathscr{N}$ is given by an $n-1$ dimensional subspace defined in terms of projections of harmonic measures. More specifically, we have
\[|G(z)|^2-H(z,z_0)\perp_{L^2(\Gg)} \mathrm{Re} A(\Gg)\]
where $H(z,z_0)$ denotes the reproducing kernel for $L^2$-integrable harmonic functions in $\Gg$. (We have $H(z,z_0)=2\mathrm{Re}k(z,z_0)-1$, where $k$ denotes the usual Bergman kernel, see \cite{DSbook}.)
Thus, we have the representation
\begin{equation}
|G(z)|^2=H(z,z_0)+\sum_{j=1}^{n-1}\lambda_j\nu_j(z)+\Delta \phi(z),
\label{bergmandecomp}
\end{equation}
where $\nu_j(z)=\omega_j(z)-P_{A^2(\Gg)}\omega_j(z)$, $P_{A^2(\Gg)}\colon L^2(\Gg)\to \mathrm{Re}(A^2(\Gg))$ being the orthogonal projection, and $\phi\in W_0^{1,2}(\Gg)$. (Recall that, by Weyl's lemma, the $L^2$ closure of $\Delta \phi$, $\phi \in C^{\infty}_0(\Gg)$, the Sobolev space $W^{1,2}_0(\Gg)$, is precisely the annihilator of square integrable harmonic functions in $\Omega$, see \cite{DKSSPac,DSbook}.)
At the moment it is not clear to us, however, how to extract a more concrete representation (e.g. in terms of an area-orthonormal basis) for $|G|$ from this.

To summarize, we have found that solutions to the extremal problems \eqref{eq:extprob1} and \eqref{eq:extprob2} do not in general have constant modulus on $\gamma_j$, and consequently, are not $H^2(\Gg)$-inner in the sense of Definition \ref{multinnerconstants}. For the Hardy space at least, we have an explicit representation of the defect space as the span of Schottky functions,
\[\mathscr{N}=\mathrm{Harm}(\overline{\Gg})\ominus_{L^2(\Gg, d\omega)} \mathrm{Re}[A(\Gg)]=\mathrm{span}\{s_j,\,j=1,\ldots, n-1\},\] 
but in the Bergman space, a representation of $\mathscr{N}=\mathrm{Harm}(\overline{\Gg})\ominus_{L^2(\Gg)} \mathrm{Re}[A(\Gg)]$ that is equally transparent is not currently available to us.

As was mentioned in Section \ref{InnerDefs}, an important reason for studying inner functions is that they arise as generators of invariant subspaces (at least in instances where these invariant subspaces are singly generated). This is the content of Beurling's invariant subspace theorem for $H^p$. Obtaining such generators by solving extremal problems was the approach taken by Hedenmalm in his identification of contractive divisors for $A^2$, and these contractive divisors were later identified as Bergman-inner functions. While solutions to extremal problems in multiply connected domains need not be inner functions in the sense of Definition \ref{multinnerconstants}, we have not yet explained how extremals relate to generators of invariant subspaces, in particular in $H^2(\Gg)$, where an analog of Beurling's theorem is known to hold. 
We will return to this discussion in Section  \ref{InvariantSubs}. 

First, we need to achieve a better understanding of solutions to extremal problems in the multiply connected setting. A basic obstruction is the presence of extraneous zeros: these are points in $\Gg\setminus \{z_1\}$ where $G$, the solution to the extremal problem \eqref{eq:extprob1}, vanishes. Even in radially weighted Bergman spaces in the unit disk, extraneous zeros may be present \cite{HZ,Weir}, and in the multiply connected setting, the Bergman kernel function itself has zeros whose number depend on the connectivity of $\Gg$. 

Let us discuss reproducing kernels next. Since point evaluation is always a bounded linear functional for the Hardy space $H^2(\Gg)$ and the Bergman space $A^2(\Gg)$, making these spaces reproducing kernel Hilbert spaces, we might hope to use kernel methods to study inner functions.

We start, first, with the Smirnov space $E^{2}(\Omega)$, although we do assume that $\Gamma$ is smooth. It is know, see \cite[p. 35]{SKhav3}, that the reproducing kernel for $E^{2}(\Omega)$
is given by the expression
\begin{equation}
K(z,\zeta)=F_{\zeta}^{*}(z)L(z,\zeta), \quad \zeta \in \Omega.
\label{repkernelformula}
\end{equation}
Here, $F_{\zeta}^{*}$ is the extremal function in the Schwarz lemma
problem
\[
\sup\left\{ \left|f'(\zeta)\right|:\,f\in H^{\infty}\left(\Omega\right),\,\left\Vert f\right\Vert _{\infty}\le1\right\} ;
\]
the function $F^*$ is also known as the Ahlfors function \cite{SKhav3,Bellbook}. The second factor in \eqref{repkernelformula} is given by
\[
L\left(z,\zeta\right)=\frac{1}{2\pi}\frac{\Phi\left(z,\zeta\right)}{z-\zeta},
\]
where 
\[
\Phi\left(z,\zeta\right)=\sqrt{1-2\pi i\left(z-\zeta\right)^{2}\varphi^{*}(z)},
\]
is single-valued in $\Gg$ (see \cite[p.39]{SKhav3}), and $\varphi^{*}$ is the extremal function for the dual problem
\[
\inf_{\varphi\in E^{1}\left(\Omega\right)}\int_{\Gamma}\left|\frac{1}{2\pi i\left(z-\zeta\right)^{2}}-\varphi(z)\right|ds(z).
\]

It is known (cf. \cite[Thm 8.1]{SKhav3}) that $F_{\zeta}^{*}$
has $n$ zeros in $\Omega,$ including one at $\zeta$. It is
also known that $\Phi$ has no zeros in $\Omega$ (again, cf. \cite[pp. 39-40]{SKhav3}).
It follows then, by the definitions of $F_{\zeta}^{*}$ and $L$ that
for any fixed $\zeta\in\Omega$, the kernel $K(z,\zeta$) has precisely
$n-1$ zeros in $\Omega$. We now prove the following proposition.
\begin{proposition}
\label{proposition:Weighted Smirnov}For any non-negative $\rho$ having
$\rho\in L^{1}(\Gamma,ds)$ and  $\log \rho \in L^1(\Gamma,ds)$, and any fixed $\zeta\in\Omega$, the reproducing
kernel $k_{\rho}(z,\zeta)$ of $E^2(\rho ds)$ has precisely $n-1$
zeros in $\Omega$.
\end{proposition}
\begin{proof}
First assume that $\rho=\left|g\right|^{2}$, where $g\in A(\Omega)$
does not vanish in $\overline{\Omega}$. Then, trivially, we have  
\[
k_{\rho}(z,\zeta)=\frac{K(z,\zeta)}{g(z)\overline{g(\zeta)}}.
\]
Indeed, for any $f\in H^{\infty}$,
\begin{align*}
 \int_{\Gamma}f(\zeta)k_{\rho}(z,\zeta)\rho(\zeta)ds(\zeta)
&=  \int_{\Gamma}f(\zeta)\frac{K(z,\zeta)}{g(z)\overline{g(\zeta)}}g(\zeta)\overline{g(\zeta)}ds(\zeta)\\&=  \frac{1}{g(z)}f(z)g(z)\\&=f(z).
\end{align*}
Since $H^{\infty}$ is dense in $E^2(\rho ds)$, this proves the claim.

Now, by \cite[Thm. 4]{Khav2}, any function on $\Gamma$ satisfying $\rho(\zeta)\ge 0$ for $\zeta \in \Gamma$ and having $\rho, \log\rho\in L^{1}(\Omega)$
is equal almost everywhere to $\left|g\right|^{2}$ for some function $g\in H^{2}(\mathcal{G})$
that does not vanish in $\Omega$. The same argument as above
then applies. Moreover, it is clear then that for any fixed $z_0\in\Omega$,
the reproducing kernel $k_{\rho}(z,z_0)$ of $E^{2}(\rho ds)$ has
precisely $n-1$ zeros in $\Omega$.
\end{proof}
\begin{remark}
Using more recent results in \cite{SKhav1,SKhav2}, we believe that the
above proposition can be generalized to finite Riemann surfaces.
\end{remark}

It is known, see for instance \cite{GS}, that the reproducing
kernel for the Bergman space $A^{2}(\Gg)$ has exactly the
same number of zeros as the critical points of the Green's function: in an $n$-connected domain, there are $n-1$ zeros.
If we consider a weight $\rho=\left|g\right|^{2}$
with $g\in A^{2}(\Gg)$ non-vanishing in $\overline{\Gg}$, exactly the same argument as above extends Proposition \ref{proposition:Weighted Smirnov}
to weighted Bergman spaces with such weights. Of interest to us
is what happens when the function $g$ is allowed to have zeros in
$\Gg.$ In that case, we have the following proposition.
\begin{proposition}
\label{proposition:Weighted Bergman}For any logarithmically subharmonic weight of the form $\rho=\left|g\right|^{2}$
with $g$ analytic in $\overline{\Gg}$, and any fixed $z_0\in\Gg$,
the reproducing kernel $k_{\rho}(z,z_0)$ of $A_{\rho}^{2}(\Gg)$
has precisely $n-1$ zeros in $\Gg$, where $n$ is the number
of connected components of $\Gamma$.
\end{proposition}
\begin{proof}
As mentioned, when $g$ is nonvanishing in $\overline{\Gg}$, we can
extend Proposition \ref{proposition:Weighted Smirnov} \emph{mutatis mutandis.
}Suppose now that $g$ does vanish in $\Gg$ (but not on $\Gamma$), and let $z_{1},\ldots, z_{m}$
be the zeros of $g$ in $\Gg$. For $\epsilon>0$, let $\Delta_{j,\epsilon}$
be the disk of radius $\epsilon$ centered at $z_{j}$, and let
\[\Gg_{\epsilon}=\Gg\setminus\cup_{j=1}^{m}\Delta_{j,\epsilon}\]
be the domain obtained by excising these disks. This domain then is
$n+m$-connected and so its unweighted Bergman kernel $k_{\epsilon}(z,z_0)$
has $n+m-1$ zeros in $\Gg_{\epsilon}$. Now 
the weight $\rho=\left|g\right|^{2}$ is nonvanishing in $\Gg_{\epsilon}$,
and so the kernel of the weighted Bergman space $A_{\rho}^{2}(\Gg_{\epsilon})$
is given by 
\[
k_{\rho, \epsilon}(z,z_0)=\frac{k_{\epsilon}(z,z_0)}{g(z)\overline{g(z_0)}},
\]
and hence has $m+n-1$ zeros in $\Gg_{\epsilon}$. 

Letting $\epsilon\rightarrow0$, we have that 
\begin{equation}
k_{\rho, \epsilon}(z,z_0)\rightarrow k_{\rho}(z,z_0), \quad z\notin\bigcup_{j=1}^{m}\Delta_{j,\epsilon}.
\label{shrinkingholes}
\end{equation}

Since $\Gg \ni z\mapsto k_{\rho}(z,z_0)$ is an analytic function, it must
be the case that $m$ zeros of $k_{\epsilon}(z, z_0)$ tend to the zeros of $g$ as $\epsilon \to 0$ in order to cancel
out the poles at $z_1,\ldots, z_m$. But then 
$k_{\rho}(z,z_0)$ has $n+m-1-m=n-1$ zeros, as claimed, and in light of \eqref{shrinkingholes} and Hurwitz's theorem, these zeros coincide with those of the unweighted reproducing kernel $k(z,z_0)$. 
\end{proof}
The argument here is in the spirit of \cite{KS}. The question
of whether the number of zeros for $k_{\rho}$ is $n-1$ for {\it any}
logarithmically subharmonic weight $\rho$ (see \cite{Hbook,DSbook} for a definition)  is still open. 

We return to solutions of the extremal problem \eqref{eq:extprob1} and draw some conclusions from the facts about reproducing kernels established above. First, let $B_{z_1}$ denote a ``Blaschke factor" for $\Gg$ with a single zero at $z_1$, that is, $B_{z_1}$ maps $\Gg$ to a disk with $n-1$ concentric circular slits removed and has $B_{z_1}(z_1)=0$, see  \cite{CW,Khav1}. Then one checks that the extremal function for \eqref{eq:extprob1} can be realized as
\[G(z)=B_{z_1}(z)k^{|B_{z_1}|^2}(z,z_0),\]
where $k^{|B_{z_1}|^2}$ denotes the normalized reproducing kernel in a space with weight given by $|B_{z_1}|^2$. In the case of the Bergman space, we weight area measure by $|B_{z_1}|^2$, while in the case of the Hardy space we consider the norm inherited from $L^2( \Gamma,|B_{z_1}|^2d\omega_{z_0})$, and in both cases, $\|k^{|B_{z_1}|^2}\|_{\mathscr{X}}=1$. Now, by Propositions \ref{proposition:Weighted Smirnov} and \ref{proposition:Weighted Bergman}, the weighted $k^{|B_{z_1}|^2}$ and the unweighted kernel $k(\cdot, z_0)$ have the same extraneous zeros.
This property is inherited by the extremal function $G$. Moreover, the fact that extraneous zeros of extremals coincide with those of the reproducing kernel remains unaffected if we replace the problem \eqref{eq:extprob1} by the more general problem
\[\sup\left\{ \left|f(z_{0})\right|\colon f(z_{j})=0 \quad \textrm{for} \quad z_j \in \mathcal{Z} \quad \textrm{and} \quad \left\Vert f\right\Vert _{\mathscr{X}}=1\right\},\]
for a Hardy or Bergman-zero set $\mathcal{Z}$.

\section{Inner functions and invariant subspaces}\label{InvariantSubs}
We have seen that translating the usual orthogonality conditions and extremal problems that provide equivalent definitions for inner functions in simply connected domains do not immediately 
lead to inner functions in the sense of Definition \ref{multinnerconstants}. We now review some of the good properties of 
functions that are inner in the sense of Definition \ref{multinnerconstants}, and we also make some new observations.

We follow Royden \cite{Roy} and say that a closed subspace $S\subset H^2(\Gg)$ is fully invariant if $rS\subset S$
whenever $r\in R(\Gg)$. Zero-based subspaces 
\[\mathscr{M}_{\mathcal{Z}}=\{f \in\mathscr{X}\colon f(z_j)=0, \,\, z_j \in \mathcal{Z}\}\] 
are prime examples of such fully invariant subspaces.

We now state the analog of Beurling's theorem for multiply connected domains. This theorem traces its origins back to the 1960s \cite{Voi,VZ,Has}, and has inspired numerous subsequent papers, see for instance \cite{Y,AleRic,ChenEtal}. Further operator-theoretic developments can be found in \cite{AHRbook}.
\begin{theorem}\label{multiconnbeurling}
Let $\mathscr{M}\neq \{0\}$ be a closed fully invariant subspace of $H^2(\Gg)$. Then there exists an inner function $G$ such that $\mathscr{M}=GH^2(\Gg)$.
\end{theorem}

While proofs of Theorem \ref{multiconnbeurling} can be found in several cited sources, the argument proceeds along slightly different lines. Following the streamlined proof in Hoffman \cite{Hoffbook}, we make systematic use of factorization of Hardy functions, making the proof quite similar to that for the disk case. As in \cite{Roy}, the proof is readily extended to cover invariant subspaces in $H^p$,  $p\geq 1$. (In fact, extensions to the case $p<1$ are also possible after a suitable modification of the notion of invariant subspace, but we shall not pursue this here.)
\begin{proof}
Define the function
\[G_0=P_ {\mathscr{M}}[1],\]
where $P_{\mathscr{M}}\colon H^2(\Gg)\to \mathscr{M}$ denotes the orthogonal projection onto the subspace $\mathscr{M}$. Pick $z_0\in \Gg$ that is not a joint zero of all elements on $\mathscr{M}$. Then, for all $r \in A(\Gg)$ having $r(z_0)=0$,
\[0=\int _{\Gamma}(1-\overline{G}_0)G_0rd\omega_{z_0}=-\int_{\Gamma}|G_0|^2rd\omega_{z_0}.\]
Set $G=G_0/\|G_0\|_{H^2(\Gg)}$. Then $G\in \mathscr{M}$, and moreover
\[\int_{\Gamma}(|G|^2-1)rd\omega_{z_0}=0 \quad \textrm{for}\quad r\in A(\Gg).\]

It remains to prove that $G$ generates $\mathscr{M}$. By Proposition \ref{Schottkyprop},
\begin{equation}
|G(\zeta)|^2=1+\sum_{j=1}^{n-1}\lambda_js_j(\zeta),
\label{Gsum}
\end{equation}
where $s_j$ are the Schottky functions for $\Gg$. By \cite[Section 4]{Khav1} (see also \cite{CW}), we have
\[G=Iu,\]
where $I$ is inner, and $u$ is non-vanishing in $\Gg$, and analytic in $\overline{\Gg}$ in light of \eqref{Gsum}. Hence $u$ is outer, with at most finitely many zeros on $\Gamma$, and hence $u$ is cyclic. (This follows from an argument that shows that any rational function with no zeros in $\Gg$ and with at most finitely many zeros on $\Gamma$ is cyclic; we sketch the argument. First, we establish cyclicity for functions of the form $f=(z-\zeta_0)\cdot v$,  $\zeta_0\in \gamma_1$ fixed and $v$ non-vanishing on $\overline{\Gg}$, in a smaller Hardy space on $\mathrm{int}(\gamma_1)$. In that case, as in the unit disk, we can find a sequence of polynomials $\{p_n\}$ such that $p_nf-1\to 0$ in $L^2(\gamma_1)$, and by subharmonicity, the same holds for other boundary components, and this then shows that $(z-\zeta_0)v$ is cyclic. Finally, we use that a finite product of cyclic multipliers is cyclic to extend the argument to the case of finitely many zeros on $\Gamma$.)

Now suppose $f \in \mathscr{M}$ and $f\perp GH^2(\Gg)$ and recall that $G=P_{\mathscr{M}}[1]\|G_0\|$. Then
\[\int_{\Gamma}\overline{f}Grd\omega_{z_0}=0, \quad r\in A(\Gg),\]
and for $r\in A(\Gg)$ with $r(z_0)=0$, we also have 
\[\int_{\Gamma}\overline{G}frd\omega=0.\]
Together, this implies that $\overline{f}G\perp (A(\Gg)\oplus \overline A_0(\Gg))$. Then, since the Schottky functions $s_j$ are real-valued on $\Gamma$, we deduce that
\begin{equation}
\overline{G}f \in \mathscr{S}=\mathrm{span}_{\mathbb{C}}\{s_1,\ldots, s_{n-1}\} 
\label{inschottky}
\end{equation}
and also \[G\overline{f}\in \mathscr{S}.\]
Since $|G|$ is real-analytic on $\Gamma$, and elements of $\mathscr{S}$ are real-analytic on the boundary, the function $|f|$ is also real-analytic. 

Consider the factorizations
\[G=I_GF_GQ_G \quad \textrm{and}\quad f=I_fF_fQ_f,\]
where the functions $I_G$ and $I_f$ are multi-valued functions, while $F_G$ and $F_f$ are single-valued outer factors, analytic across $\Gamma$ in light of \eqref{Gsum} and \eqref{inschottky}, and $Q_G$ and $Q_f$ are ``period removers" of the form $Q=\exp[\sum_{j=1}^{n-1}\lambda_j(\omega_j+i\widetilde{\omega}_j)]$, for suitably chosen $\lambda_j$. From \eqref{inschottky}, we have
\[\overline{G}f=(I_f/I_G)\,Q_f\overline{Q}_{G}\,\overline{F}_GF_f\]
and since $Q_{f}, Q_{G}$ and $\overline{F}_G, F_f$ are real-analytic on $\Gamma$, and $F_G$ and $F_f$ are single-valued, $(I_f/I_G)Q_f\overline{Q}_G$ is single-valued and real-analytic near $\Gamma$. This in turn implies that $\log^+|I_f|-\log^+|I_G|$ is real-analytic on $\Gamma$ as well, and, by uniqueness for Smirnov classes \cite{Khav1,SKhavTum3}, the singular measures appearing in $I_f$ and $I_G$ coincide.

We now conclude from this that $f$ and $G$ have the same zeros in a neighborhood of $\Gamma$. This means we can write
\[Q_fI_f=I_GQ_GB_0,\]
where $B_0$ is an inner factor with finitely many zeros in $\Gg$, analytic up to the boundary. Hence
\begin{equation}
f=I_GQ_GQ_fF_fB_0=(G/F_G)Q_fF_fB_0.
\label{ffactorization}
\end{equation}
Note that the factor $B_0Q_f$ is necessarily single-valued since the other factors are single-valued. Interchanging the roles of $f$ and $G$ shows that $B_0$ is in fact trivial, and hence $f\in [G]$ by \eqref{ffactorization}. Since also $f\perp [G]$, we have $f=0$, and this finishes the proof.
\end{proof}

Let us return to solutions to extremal problems and examine how they relate to the inner-function generators of invariant subspaces in Theorem \ref{multiconnbeurling}. If $\mathscr{M} \subset H^2(\Gg)$ is an invariant subspace, we can pose the extremal problem
\[\sup\{|f(z_0)|\colon f\in \mathscr{M}, \|f\|_{H^2(\Gg)}=1\};\]
here, we are assuming that $z_0\notin \mathcal{Z}_ {\mathscr{M}}=\{z\in \Gg\colon f(z)=0 \quad \forall f\in \mathscr{M}\}$. The extremal function solving this problem can now be written
\[G_e(z)=G(z)k^{\mathscr{M}}(z,z_0),\]
with $k^{\mathscr{M}}$ denoting the reproducing kernel  for $H^2(\Gg)$ with weight $|G|^2$, where $G$ generates $\mathscr{M}$. Note that $G_e\in \mathscr{M}$ by Theorem \ref{multiconnbeurling}. As was explained in the previous section, the zeros of $k^{\mathscr{M}}$ coincide with the zeros of $k$, the Szeg\H{o} kernel, which is reproducing for $H^2(\Gg)$. Let
\[Z=\{w_1(z_0), \ldots, w_{n-1}(z_0)\} \subset \Gg\]
denote these zeros, and let $G_{Z} \in H^{\infty}(\Gg)$ be a fixed inner function with these zeros. In other words, $|G_{Z}(\zeta)|_ {\big|_{\gamma_j}}=c_j$, we have $G_{Z}(w_k)=0$ and $G_Z$ has no other zeros in $\Gg$, and the singular inner factor of $G_Z$ is identically equal to $1$. After dividing $G_e$ by $G_Z$, we obtain another generator for $\mathscr{M}$, and thus arrive at an alternative proof of Theorem \ref{multiconnbeurling}.  

We turn to the problem of dividing out zeros. Note that, unlike in the disk, we cannot hope for isometric divisors in the case of the Hardy space $H^2(\Gg)$. Indeed, if $G$ were an isometric divisor for $H^2(\Gg)$, then we would have  $|G|=1$ on $\Gamma$. If $G\in C(\overline{\Gg})$ is not constant, then $G'(\zeta)\neq 0$ on $\Gamma$, and hence $\partial (\arg G)/ds>0$. Hence $\arg G$ increases by an integer multiple of $2\pi$ along each $\gamma_j$, and the argument principle then implies that $G$ has at least $n$ zeros. But this makes it impossible for $G$ to be a divisor for invariant subspaces associated with fewer than $n$ zeros in $\Gg$.

We recall that a contractive divisor associated with a zero set $\mathcal{Z}$ for $A^p(\mathbb{D})$, a Bergman space in the disk, is a function $G$ that vanishes on $\mathcal{Z}$, and satisfies $\|f/G\|_{A^p}\leq \|f\|_{A^p}$ for any $f\in \mathscr{M}_{\mathcal{Z}}$. Guided by this, we now make the following definition for a function space $\mathscr{X}$ in a multiply connected domain.
\begin{definition}
Let $\mathcal{Z}$ be a $\mathscr{X}$-zero set. We say that $G\in \mathscr{X}$ is a quasi-contractive divisor if $G\in \mathscr{M}_{\mathcal{Z}}$ and there exists a constant $C_{\mathscr{X}, \Gg}$ such that
\[\|f/G\|_{\mathscr{X}}\leq C_{\mathscr{X}, \Gg}\|f\|_{\mathscr{X}}, \quad f\in \mathscr{M}_{\mathcal{Z}}.\]
\end{definition}

We now prove that $H^2(\Gg)$ supports quasi-contractive divisors. 
\begin{theorem}\label{QCdiv}
Let $\mathcal{Z}=\{z_j\}_{j=1}^{\infty}$ be an $H^2(\Gg)$-zero set and let $\mu$ be a non-positive measure with $\mu \perp \omega_{z_0}$.  

There exist positive numbers $\{\lambda_j\}_{j=1}^{n}$
such that the inner function
\[G=B_{\mathcal{Z}}\cdot S_\mu,\]
with $B_{\mathcal{Z}}$ and $S_{\mu}$ as in \eqref{blaschkeformula} and \eqref{singularformula},
satisfies 
\[|G(\zeta)|_{\Big|_{\gamma_j}}=e^{\lambda_j}  \quad \textrm{and} \quad 1\leq \|G\|_{H^2}\leq e^{\lambda_j}.\]

Moreover, if $G_0=G/\|G\|_{H^2(\Gg)}$ and $c_{\Gg}=\max_{j=1,\ldots, n} \lambda_j$, then $0<c_{\Gg}<\infty$, and setting $C_{\Gg}=e^{c_{\Gg}}$, we have
\[C_{\Gg}^{-1}\|f\|_{H^2(\Gg)}\leq \|G_0f\|_{H^2(\Gg)}\leq C_{\Gg}\|f\|_{H^2(\Gg)}, \quad f\in H^2(\Gg).\] 
In particular, $\|f/G_0\|_{H^2}\leq C_{\Gg}\|f\|_{H^2}$ for any $f\in[G]$.
\end{theorem}
As is the case for Theorem \ref{multiconnbeurling}, Theorem \ref{QCdiv} extends to $H^p$ for $p\geq 1$ (and even to $p<1$) via factorization techniques from \cite{Khav1,Khav2}. 
\begin{proof}
 Recall that any inner function in $H^2(\Gg)$ can be written as a product of a Blaschke product and a singular inner function (and this is uniquely determined up to invertible factors) The existence of an inner function $G$ follows from the results of \cite[Lemma 2]{Khav2},  cf. Examples \ref{blaschkeex} and \ref{singularex}.
The main observation is that there exists a constant $c_{\Gg}$ with the property that, for $0< \lambda_j\leq c_{\Gg}$, $j=1, \ldots, n$, the periods of
\[\sum_{j=1}^{n}\lambda_j(\omega_j+i\widetilde{\lambda}_j)\]
around the boundary curves $\gamma_k$ cover the set 
\[\mathcal{C}=\{x \in  \mathbb{R}^{n}\colon 0\leq x_j\leq 2\pi, \, j=1,\ldots, n\}.\]

This means that, given the zero set $\mathcal{Z}$ and the singular measure $\mu$, 
we can construct a period remover 
\[Q(z)=\exp\left[\sum_{j=1}^{n}\lambda_j\left(\omega_j+i\widetilde{\omega}_j\right)\right]\]
with $\lambda_1, \ldots, \lambda_{n}$, $0<\lambda_j<c_{\Gg}$, selected in such a way that the function
\[G(z)=\exp\left(-\sum_{j=1}^{\infty}p(z_j,z_0)\right)\exp \left[\frac{1}{2\pi}\int_{\Gamma}\left(\frac{\partial g}{\partial n_{\zeta}}(z,\zeta)+i\frac{\partial \widetilde{g}}{\partial n_{\zeta}}(z,\zeta)\right)d\mu(\zeta)\right]Q(z)\]
is single-valued and inner. Now, invoking standard results about Green-Stieltjes integrals, see \cite{Khav1}, we observe that
\[1\leq |G(\zeta)|\leq \max_j e^{\lambda_j} \leq e^{c_{\Gg}}, \quad \textrm{for} \quad \omega-\textrm{a.e.}\, \quad \zeta \in \Gamma.\]
Setting $C_{\Gg}=e^{c_{\Gg}}$, we immediately obtain the first part of the statement.

To get the second part, note that $C_{\Gg}\geq \|G\|_{H^2(\Gg)}\geq 1$, meaning that $1/C_{\Gg}\leq|G_0(\zeta)|\leq C_{\Gg}$ on $\Gamma$. The claimed norm inequality follows.
\end{proof}

Since the Bergman space $A^2(\mathbb{D})$ has contractive divisors, it is natural to ask whether there are quasi-contractive divisors for $A^2(\Gg)$. We have not been able to settle this question; instead, we will report on some hitherto unsuccessful attempts to make progress on this problem, and why there are good reasons to believe that the problem might be challenging.

A natural starting point would be to consider the extremal problem \eqref{eq:extprob1}, and to view solutions as candidates for contractive or quasi-contractive divisors. However, as we have seen, the extremals $G$ have extraneous zeros, so that quotients $f/G$, where $f\in \mathscr{M}_{\mathcal{Z}}$ for some Bergman zero set $\mathcal{Z}$, become meromorphic in general. This obstruction was observed some time ago by Hedenmalm and Zhu \cite{HZ} in the context of weighted spaces in the disk as well as in multiply connected domains; see also \cite{Weir}.

Nevertheless, it is conceivable that a modified construction could work. Let us restrict to the case of a single zero for simplicity. First of all, recall from Section \ref{ExtremalProbs} that the problem
\[\sup\{\mathrm{Re} f(z_0)\colon f(z_1)=0, \, \|f\|_{A^2(\Gg}\leq 1\}, \quad z_1\in \Gg\setminus \{z_0\}\]
gives rise to extremals of the form
\[G(z)=B_{z_1}(z)k_{z_0}^{|B_{z_1}|^2}(z),\]
where $B_{z_1}$ is a Blaschke factor with a single zero at $z_1$ and constant modulus on the boundary curves $\gamma_j$, and $k_{z_0}^{|B_{z_1}|^2}$ is the normalized reproducing kernel at $z_0$ for the weighted Bergman space
$A^2(|B_{z_1}|^2dA, \Gg)$, with norm furnished by
\[\|f\|^2_{A^2(|B_{z_1}|^2)}=\int_{\Gg}|f(z)|^2\,|B_{z_1}|^2dA(z).\]
Now, by Proposition \ref{proposition:Weighted Bergman}, the kernel $k_{z_0}^{|B_{z_1}|^2}$ has $n-1$ zeros, and these zeros coincide with those of the unweighted Bergman kernel $K_{z_0}$, which only depends on the underlying domain $\Gg$. Let $B_{\Gg}$ be a Blaschke product formed from the zeros of $K_{z_0}$, and set
\begin{equation}
\tilde{G}=B_{z_1}\frac{k^{|B_{z_1}|^2}_{z_0}}{B_{\Gg}}.
\label{onezerocand}
\end{equation}
The function $\tilde{G}$ is analytic in $\Gg$, with the additional property that $\tilde{G}(z_1)=0$ and $\tilde{G}(z)\neq 0$ for $z\in \Gg\setminus \{z_1\}$, and thus $\tilde{G}$ seems like a candidate to be a divisor. The same modification can be performed for extremal functions associated with finite zero sets $\{z_1, \ldots, z_n\}$.
It is clear that there are some issues here, however: we would then like to send $n\to \infty$ to extend our construction to infinite $A^2(\Gg)$ zero sets. Unfortunately, infinite Blaschke products associated with Bergman zero sets do not converge in general unless the zero set itself happens to be a Blaschke sequence, and it is then not clear that modifications of the form above converge. Presumably, one would need tight control over the interplay between weighted kernels and Blaschke products in order to obtain convergent candidate divisors. In principle, one could also try 
\[\tilde{G}^{\dagger}=B_{z_1}\frac{k^{|B_{z_1}|^2}_{z_0}}{k_{z_0}},\]
where $k_{z_0}$ is the unweighted Bergman kernel at $z_0\in \Gg$; the ratio of $B_{\Gg}$ and $k_{z_0}$ is bounded above and below for finite sets zeros, but one could imagine that one or the other of $\tilde{G}$ and $\tilde{G}^{\dagger}$ might be easier to work with in practice.

Returning to the simplest case of a single zero and the function $\tilde{G}$ in \eqref{onezerocand}, it is not at all clear how to establish quasicontractivity. In the unit disk, the standard route to establishing that the solution to the corresponding extremal problem is a contractive divisor is via the biharmonic Green's function \cite{DKSSPac,DSbook}. Ignoring the Blaschke factor $B_{z_1}$ in the denominator of $\tilde{G}$ for a moment, recalling the representation of the extremal $G$ from \eqref{bergmandecomp}, using the fact that $H+\sum_1^n\lambda_j\nu_j$ is harmonic, and arguing as in the disk case, we find that 
\begin{multline*}
\int_{\Gg}\left(|G(z)|^2-H(z,z_0)-\sum_{1}^{n-1}\lambda_j\nu_j(z)\right)|f(w)|^2dA(z)\\=\iint_{\Gg\times \Gg}\left[g_{\Delta^2}(z,w)\Delta(|G|^2)(z)\Delta(|f|^2)(w)\right]dA(z)dA(w) \quad \textrm{for}\quad  f\in H^{\infty}(\Gg).
\end{multline*}
Here, $g_{\Delta^2}(z,w)$ is the biharmonic Green's function for $\Gg$, that is, 
\[\left\{\begin{array}{cc}\Delta_z^2g_{\Delta^2}(z,w)=\delta_{w} & \textrm{in}\quad  \Gg\\ g_{\Delta^2}=\nabla g_{\Delta^2}=0  &  \textrm{on}\quad \Gamma
\end{array}\right. .\]
When $\Gg$ is the unit disk, one now uses subharmonicity along with the crucial fact that the biharmonic Green's function $g_{\Delta^2}$ for the disk is positive to finish the proof that $G$ is a contractive divisor. 

If $\Gg$ is a multiply connected domain, then $g_{\Delta^2}$ changes sign in general, and it is no longer readily apparent that the right-hand side is non-negative. Nevertheless, it is conceivable that the integral on the left is still non-negative, and this would leave a path to proving quasicontractivity. We have not been able to advance along this route.

In conclusion, we pose a  basic problem we cannot at present solve.
\begin{problem}\label{oneptproblem}
Consider the Bergman space $A^2(\Gg)$ on a multiply connected domain $\Gg$ and fix a point $z_1\in \Gg$. Does there exist a function $G$ and a constant $C_{\Gg}$, depending on $C_{\Gg}$ but not on $z_1$, such that $G(z_1)=0$ and 
\[\|f/G\|_{A^2(\Gg)}\leq C_{\Gg}\|f\|_{A^2(\Gg)}\]
for every $f\in A^2(\Gg)$ having $f(z_1)=0$?

In particular, do such a function and such a constant exist for the Bergman space of the annulus $\Gg=\{r<|z|<1\}$?
\end{problem}
The more general version of Problem \ref{oneptproblem} is the following.
\begin{problem}\label{multiptprob}
Suppose $\Gg$ is a finitely connected domain bounded by analytic curves. Does $A^p(\Gg)$ support quasicontractive divisors? That is, given an $A^p(\Gg)$ zero set $\mathcal{Z}$, does there exist a function $G$ and a constant $C_{\Gg, p}$ such that 
\[\|f/G\|_{A^p(\Gg)}\leq C_{\Gg,p}\|f\|_{A^p(\Gg)}\]
for every $f\in \mathscr{M}_{\mathcal{Z}}$?
\end{problem}

Another possible way of attacking Problems \ref{oneptproblem} and \ref{multiptprob} in $A^2(\Gg)$ is via uniformization, briefly alluded to in the Introduction. By our assumptions on the domain $\Gg$, its universal covering space can be identified with the unit disk $\mathbb{D}$. If $\tau \colon \mathbb{D}\to \Gg$ denotes the corresponding uniformizing map, we say that a conformal map $m\colon \mathbb{D}\to \mathbb{D}$ is a deck transformation if $\tau \circ m=\tau$. The set of all deck transformations form a group under composition, which we denote by $\mathrm{Aut}(\tau)$.

Now let $\tau \colon \mathbb{D}\to \Gg$ be a uniformizing map, and consider the weighted Bergman space $A^2(\mathbb{D},|\tau'|^2dA)$. If we could establish the existence of quasicontractive divisors in this space, then we would be able to transfer back to $A^2(\Gg)$. For weighted Bergman spaces in the disk with logarithmically subharmonic weights, these questions have been addressed in a number of papers, see \cite{Hbook,DSbook,Weir} and the references therein. However, a significant obstruction manifests itself.
\begin{problem}
Suppose $\mathcal{Z}=\{z_j\}_{j=1}^{\infty}$ is an $A^2(\mathbb{D},|\tau'|^2dA)$-zero set, automorphic with respect to $\mathrm{Aut}(\tau)$, the group of deck transformations corresponding to the uniformizing map $\tau$. 
Let $G$ be the contractive divisor for $A^2(\mathbb{D})$ that is associated with $\mathcal{Z}$. 

Is $G$ automorphic with respect to $\mathrm{Aut}(\tau)$?
\end{problem}
We suspect the answer should be positive, but we do not have a proof.

Finally, we mention yet another approach to Problems \ref{oneptproblem} and \ref{multiptprob} that involves reduction to simply connected domains. 

\begin{problem}\label{productprob}
Let $\mathcal{Z}=\{z_j\}_{j=1}^{\infty}$ be an $A^2(\Gg)$-zero set. Let $A^2(\Omega_j)$ be the Bergman space on the simply connected domain $\Omega_j\supset \Gg$ bounded by $\gamma_j$. 

Can we write $\mathcal{Z}=\bigcup_{j=1}^n \mathcal{Z}_j$, where $\mathcal{Z}_{j}\cap \mathcal{Z}_{k}=\emptyset$, and each $\mathcal{Z}_j$ is an $A^2(\Omega_j)$-zero set? If this were the case, and if we let $G_j$ denote the contractive zero divisor for $A^2(\Omega_j)$, then $G=\prod_{j=1}^{n}G_j$ would solve Problem \ref{multiptprob}. 
\end{problem}
In order to resolve Problem \ref{productprob}, we should be able to answer the following question.
\begin{problem}
Can every $f\in A^2(\Gg)$ be factored as 
\[f=\prod_{j=1}^nf_j, \quad f_j \in A^2(\Omega_j)\]
where \[c_{\Gg}\|f\|_{A^2(\Gg)}\leq \|f_j\|_{A^2(\Omega_j)}\leq C_{\Gg}\|f\|_{A^2(\Gg)}\]
for each $j=1, \ldots,n$, with $0<c_{\Gg}\leq C_{\Gg}<\infty$.
\end{problem}
The corresponding problem for sums, that is, writing $f=\sum_{j=1}^ng_j$ with $g_j\in A^2(\Omega_j)$ and $\|f\|_{A^2(\Gg)}\sim \|g_j\|_{A^2(\Omega_j)}$, can be seen to admit a solution by splitting 
the Cauchy integral formula for $f$ on a regular exhaustion of $\Gg=\bigcup_j\Omega_j$. See \cite{Durbook,SKhavTum1,SKhavTum4} for similar questions in the settings of Hardy and Smirnov spaces. Problem \ref{productprob}
involves applying the Cauchy integral formula to logarithms that are not single-valued. Nevertheless, we believe Problem \ref{productprob} should admit a positive resolution.

\section{Extension to rough domains and Riemann surfaces}
The Hardy spaces $H^p$, $p>0$, as well as the Nevanlinna and Smirnov classes $N$ and $N_+$, are defined in terms of harmonic measures. Hence, they are conformally invariant, as are Schottky functions, cf. \cite{Khav1}. Therefore, all the observations and results proved for these spaces can be extended {\it mutatis mutandis} to finitely connected domains with arbitrary Jordan boundaries by using factorization theorems as in \cite{Khav1}. 

However, the case of the Smirnov spaces $E^p$, $p>0$, is much more delicate. In domains with smooth boundaries, or more generally, in domains such that the conformal map
\[\psi \colon \Gg \to K\]
of the given domain $\Gg$ onto a canonical circular domain satisfies 
\begin{equation}
0<c_1\leq |\psi'(z)|\leq c_2<\infty
\label{confmapbound}
\end{equation}
for some constants $c_1,c_2$, the Smirnov spaces $E^p$ coincide with the Hardy spaces $H^p$ as sets \cite{Durbook,SKhavTum1}. Moreover, there is a simple isometric isomorphism of the form
\[f\mapsto f (\psi')^{1/p}\]
between $H^p(\Gg)$ and $E^p(\Gg)$ provided by the Keldysh-Lavrentiev theorem (again, see \cite{Durbook,SKhavTum1}). Thus,  Theorems \ref{multiconnbeurling} and \ref{QCdiv} as well as Propositions 1-3 extend to $n$-connected domains $\Gg$ for which condition \eqref{confmapbound} holds.

It should be noted that $E^p$-classes are of great importance in more general $n$-connected domains, where the boundary is merely assumed to be comprised of rectifiable Jordan curves. The reason is that by well-known generalizations of the F. and M. Riesz theorem \cite{SKhavTum1,Durbook,Golbook}, $E^p$-functions ($p\geq1$) in such domains are still Cauchy integrals of functions in $L^p(\Gamma, ds)$. One could therefore argue that $E^p$ spaces are ``more natural" than their Hardy counterparts in domains with rectifiable boundaries. However, in this generality, the situation is quite complicated. 

An $n$-connected domain is said to be of Smirnov type if $\psi'\in N_{+}$, where $\psi \colon \Gg\to K$ is a conformal map as above; see \cite{Durbook,Khav0,DeCKhav1,DeCKhav2}. If $\Gg$ is a Smirnov domain, then $R(\Gg)$ is dense in each $E^p$, $p>0$, and moreover $N\supset H^p$ for all $p>0$, and hence the notion of an invariant subspace carries over. Factorization theorems and the Theorems and Propositions in this paper also extend. Yet, in domains not of Smirnov type, $E^p\not\subset N_+$ \cite{Khav1}, and $R(\Gg)$ is never dense in $E^p$, $p>0$. In fact, this can be taken as an equivalent definition of Smirnov domains, going back to V.I. Smirnov himself \cite{Durbook,Golbook}. Therefore, for a domain not of Smirnov type yet having a rectifiable boundary, it is not clear how an ``invariant subspace" should be defined. If one takes the easy way out, and defines invariant subspaces in $E^p$ as isomorphic images of invariant subspaces in $H^p$ using the isomorphism $\psi$, then again, all statements extend trivially.

Here is a more challenging question.
\begin{problem}
Find analogs of Theorems \ref{multiconnbeurling} and \ref{QCdiv} for $E^p(\Gg)$-spaces on non-Smirnov domains $\Gg$.
\end{problem}
We do not have a crisp conjecture in mind, but it appears likely that the singular factor of $\psi'$ would play a significant role.

There is another delicate point for domains of connectivity $n>1$.  Namely, suppose that some boundary components in $\Gamma=\bigcup_{j=1}^n\gamma_j$ are smooth, while others are ``rough"---can one formulate reasonable analogs of Theorem \ref{multiconnbeurling} and \ref{QCdiv} in that setting? For instance, suppose $\Gamma=\gamma_1\cup \mathbb{T}$ with $\mathrm{int}(\gamma_1)$
being a non-Smirnov domain, for instance, a ``pseudocircle" \cite{Durbook}.

For the Bergman spaces $A^p$, it is known that Theorem \ref{multiconnbeurling} fails for general invariant subspaces, even in the unit disk \cite{DSbook,Hbook}. However, the analog of Theorem \ref{multiconnbeurling} does hold for zero-based invariant subspaces in the disk: for $p=2$ this is due to Hedenmalm \cite{Hed}, and to Duren, Khavinson, Shapiro, and Sundberg for general exponents \cite{DKSSPac}. This reduced version of Beurling's theorem, as well as the analog of Theorem \ref{QCdiv} extends to simply connected domains for which \eqref{confmapbound} holds. This is also true for simply connected domains having strictly positive harmonic reproducing kernel and non-negative biharmonic Green's function. See \cite[Theorem 3]{DKSSMich}, and the remarks that follow, for details.

Now, as we mentioned earlier, it is an open problem to construct even one-point quasicontractive divisors in multiply connected domains with analytic boundaries. Relaxing boundary smoothness but decreasing connectivity, we can pose the following problem.
\begin{problem}
Suppose $\Gg$ is a simply connected domain with rough boundary $\Gamma$. (For instance, suppose $\Gg$ is such that \eqref{confmapbound} is violated for a conformal map onto the disk). Does $A^p(\Gg)$ support quasicontractive divisors? 
 \end{problem}
 
Next, we turn to finite Riemann surfaces with smooth boundaries. In this setting, S.Ya. Khavinson \cite{SKhav1,SKhav2} has developed factorization techniques that should in principle allow Theorems \ref{multiconnbeurling} and \ref{QCdiv} to be extended. To work out the technical details nevertheless seems challenging, as one needs to differentiate between removing periods around boundary components and removing periods around handles. We believe that it is a worthy task to prove extensions of Theorems \ref{multiconnbeurling} and \ref{QCdiv} for Riemann surfaces rigorously.
 
Finally, one might consider potential analogs of Theorems \ref{multiconnbeurling} and \ref{QCdiv} for Hardy and Bergman spaces in infinitely connected domains or on Riemann surfaces of infinite genus. A natural starting point here might be the so-called Parreau-Widom domains; see \cite{Heibook,P,W}. Virtually nothing is known in this setting, and $H^p$-factorization involving single-valued zero divisors is not currently available. 
In general domains or Riemann surfaces, criteria for $H^p$ or $A^p$ to be non-trivial, in the sense of containing non-constant functions, differ quite substantially for different values of $p$, see \cite{Carbook,Hej,Has2,Has3,Hed2,Bjorn}, making this type of endeavor potentially quite complicated but also rewarding.
\section*{Acknowledgements} CB and DK are grateful to Stockholm University, where part of this work was carried out, for hospitality. AS thanks the Department of Mathematics and Statistics, University of South Florida, for its warm hospitality in January 2018.  DK acknowledges support from the Simons Foundation in the form of a Collaborative Grant.

\end{document}